\documentclass[oneside,11pt]{amsart}
\usepackage{amsmath}
\usepackage{amsthm}

\usepackage[T1]{fontenc}

\usepackage[small]{caption}

\usepackage{amsfonts}
\usepackage{amssymb}
\usepackage{amsxtra}

\newtheorem{thm}{Theorem}[section]

\newtheorem{prop}[thm]{Proposition}
\newtheorem{lem}[thm]{Lemma}
\newtheorem{cor}[thm]{Corollary}

\theoremstyle{definition}
\newtheorem{defn}[thm]{Definition}
\newtheorem{rem}[thm]{Remark}

\newtheorem{exmp}[thm]{Example}

\renewcommand{\bar}[1]{\overline{#1}}
\newcommand{\boundary}{\partial}

\newcommand{\bigabs}[1]{\bigl|{#1}\bigr|}

\newcommand{\bigset}[2]{ \bigl\{ \, {#1} \bigm| {#2} \, \bigr\} }

\renewcommand{\setminus}{-}

\newcommand{\field}[1]{\mathbb{#1}}
\newcommand{\Z}{\field{Z}}
\newcommand{\R}{\field{R}}
\newcommand{\N}{\field{N}}

\newcommand{\of}{\circ}

\DeclareMathOperator{\CAT}{CAT}

\DeclareMathOperator{\diam}{diam}

\newcommand{\Set}[1]{\mathcal{#1}}

\usepackage{combelow}

\usepackage{color}
\usepackage{pinlabel}

\usepackage{ifthen}

\newcommand{\showcomments}{yes}

\newsavebox{\commentbox}
{\ifthenelse{\equal{\showcomments}{yes}}%
{\footnotemark
    \begin{lrbox}{\commentbox}
    \begin{minipage}[t]{1.25in}\raggedright\sffamily\upshape\tiny
    \footnotemark[\arabic{footnote}]}
{\begin{lrbox}{\commentbox}}}%
{\ifthenelse{\equal{\showcomments}{yes}}%
{\end{minipage}\end{lrbox}\marginpar{\usebox{\commentbox}}}
{\end{lrbox}}}

\begin{document}

\title{Hyperbolic groups and local connectivity}

\author[G.C.~Hruska]{G.~Christopher Hruska}
\address{Department of Mathematical Sciences\\
         University of Wisconsin--Milwaukee\\
         PO Box 413\\
         Milwaukee, WI 53211\\
	 USA}
\email{chruska@uwm.edu}

\author[K.~Ruane]{Kim Ruane}
\address{Department of Mathematics\\
         Tufts University\\
         Medford, MA 02155\\
	 USA}
\email{kim.ruane@tufts.edu}

\begin{abstract}
The goal of this paper is to give an exposition of some results of Bestvina--Mess on local connectivity of the boundary of a one-ended word hyperbolic group.  We also give elementary proofs that all hyperbolic groups are semistable at infinity and their boundaries are linearly connected in the one-ended case.  Geoghegan first observed that semistability at infinity is a consequence of local connectivity using ideas from shape theory, and Bonk--Kleiner proved linear connectivity using analytical methods.
The methods in this paper are closely based on the original ideas of Bestvina--Mess.
\end{abstract}

\dedicatory{Dedicated to Mike Mihalik on his 70th birthday.}



\date{\today}

\maketitle

\section{Introduction}
\label{sec:Introduction}

In the 1980s, Gromov introduced the notion of a word hyperbolic group and showed that each hyperbolic group can be compactified by attaching a certain boundary at infinity, known as the Gromov boundary \cite{Gromov87}.
Shortly afterward, Mess and Mihalik conjectured that the Gromov boundary of a one-ended hyperbolic group is always locally connected.

In a seminal article, Bestvina--Mess show that the boundary of a one-ended hyperbolic group is locally connected whenever it does not contain a cut point \cite{BestvinaMess91}.  
Subsequent work of Bowditch, Levitt, and Swarup shows that the boundary can never contain a cut point \cite{Bowditch99Treelike,Levitt98,Swarup96}, which establishes that the conjecture of Mess and Mihalik is true, \emph{i.e.}, the boundary is always locally connected.

Geoghegan observed that local connectivity implies that every hyperbolic group is semistable at infinity, using a deep theorem of shape theory due to Krasinkiewicz (see \cite{Krasinkiewicz77}).
Semistability is a foundational property central to the study of the fundamental group at infinity (see \cite{HruskaRuane_Semistable} for more background).
It has been conjectured by Geoghegan and Mihalik that every finitely presented group is semistable at infinity.

Additionally, Bonk--Kleiner use analytical methods to show that every hyperbolic group that is not virtually free contains a quasi-isometrically embedded copy of the hyperbolic plane in \cite{BonkKleiner05}. In order to establish this result, Bonk--Kleiner show that the boundary of a one-ended hyperbolic group is linearly connected, a stronger quantitative version of local connectivity that is often useful for the analytical study of boundaries.

The goal of this paper is to provide a detailed exposition of the proofs of \cite[\S 3]{BestvinaMess91} and to give simple, elementary proofs of semistability and linear connectivity using the same methods.
The essential ingredient in \cite{BestvinaMess91} is a geometric property known as $(\ddag_M)$, which may or may not hold in the Cayley graph of $G$. If this property holds for some $M$, one can locally push paths outward away from a basepoint, leading directly to both semistability and linear connectivity.  Surprisingly, if $(\ddag_M)$ were to fail for any $M>0$, then the boundary would contain a cut point, which is known to be false by the Cut Point Theorem of \cite{Bowditch99Treelike,Levitt98,Swarup96}.

In \cite{Bowditch99Boundaries}, Bowditch has extended some of the ideas related to $(\ddag_M)$ from the hyperbolic setting to the relatively hyperbolic setting.  This extension is a key step in his work on local connectivity of boundaries of relatively hyperbolic groups \cite{Bowditch99Connectedness,Bowditch01}.

Section~\ref{sec:Hyperbolic} is a review of hyperbolic spaces and their boundaries.
In Section~\ref{sec:CutPoints}, we introduce the $(\ddag_M)$ condition and use the Cut Point Theorem to show that the Cayley graph of any one-ended hyperbolic group satisfies $(\ddag_M)$ for any $M>0$.
Section~\ref{sec:Semistable} contains an elementary proof that all hyperbolic groups are semistable at infinity.
In Section~\ref{sec:LinearCon}, we use similar methods to establish the linear connectivity of the boundary of a one-ended hyperbolic group by a short, elementary argument.  We conclude by sketching the proof of Bonk--Kleiner's quasi-hyperbolic plane theorem, using linear connectivity.

\subsection{Acknowledgements}
The authors are grateful to Bob Bell and Craig Guilbault for helpful feedback. The authors also thank the referee for many insightful comments that helped to improve the exposition of this paper.
The first author was partially supported by grant \#714338 from the Simons Foundation.

\section{$\delta$--hyperbolic spaces}
\label{sec:Hyperbolic}

In this section, we summarize some well-known background on the structure of $\delta$--hyperbolic spaces and their boundaries.  We refer the reader to any of the texts \cite{GhysHarpe90,CDP90,BH99,BuyaloSchroeder07} for a more detailed introduction.

\begin{defn}[Hyperbolic spaces]
\label{def:DeltaHyperbolic}
A metric space is a \emph{geodesic space} if each pair of points in $X$ is joined by a geodesic segment.
If $p,x,y$ are points of a geodesic space $X$, the \emph{Gromov product} is the quantity
\[
   (x | y)_p = \frac{1}{2}
   \bigl( d(x,p) + d(y,p) - d(x,y) \bigr).
\]
To understand the geometric significance of the Gromov product, consider a geodesic triangle with vertices $x$, $y$, and $z$, and consider the quantities
\[
   a = (y | z)_x,
   \qquad
   b = (x | z)_y,
   \qquad \text{and} \qquad
   c = (x | y)_z.
\]
A direct calculation using the definition of the Gromov product reveals that
\[
   a + b = d(x,y),
   \qquad
   a + c = d(x,z),
   \qquad \text{and} \qquad
   b + c = d(y,z).
\]
A geodesic space $X$ is \emph{$\delta$--hyperbolic} for some $\delta \ge 0$ if for any geodesic triangle with vertices $x,y,z$ in $X$ we have the following.
If $y'\in [x,y]$ and $z'\in [x,z]$ are points with $d(x,y') = d(x,z') \le (y | z)_x$, then $d(y',z')\le \delta$.

Another way to understand this definition is as follows: if we collapse the triangle onto a tripod with legs of length $a$, $b$, and $c$, then the preimage of any point in the tripod is a set of diameter at most $\delta$.
Note that, in a $\delta$--hyperbolic space, each side of any triangle lies in the $\delta$--neighborhood of the union of the other two sides.
\end{defn}

Gromov's hyperbolicity inequality states the following.

\begin{prop}[\cite{BuyaloSchroeder07}, Prop.~2.1.2]
\label{prop:GromovInequality}
If $X$ is $\delta$--hyperbolic, then
\[
   ( x | y)_p \ge \min \bigl\{ (x| z)_p, (z| y)_p \bigr\} - \delta
\]
for any base point $p \in X$ and any $x,y,z\in X$.
\end{prop}

We first introduce the Gromov boundary as a set.

\begin{defn}[The set $\bar{X}=X \cup \boundary X$]
Geodesic rays $c,c' \colon [0,\infty) \to X$  are \emph{equivalent} if $d \bigl( c(t),c'(t) \bigr)$ remains bounded as $t \to \infty$.
The \emph{boundary} $\boundary X$ is the set of all equivalence classes of geodesic rays in $X$, and we let $c(\infty) \in \boundary X$ denote the equivalence class of the ray $c$.

We may identify $\bar{X} = X \cup \boundary X$ with the set of all equivalence classes of generalized rays, defined as follows.
A \emph{generalized ray} in $X$ is a geodesic $c\colon I \to X$, where $I$ is either an interval $[0,L]$ for some $L > 0$ or $[0,\infty)$.
In the case when $I=[0,L]$ we extend $c$ to a map $[0,\infty) \to X$ by setting $c(t)=c(L)$ for all $t \ge L$, and we refer to these as ``eventually constant'' generalized rays. The point $c(L)$ is the \emph{endpoint} of $c$ and will be denoted by $c(\infty)$.
Eventually constant generalized rays are \emph{equivalent} if they have the same endpoint.
A geodesic ray and an eventually constant generalized ray are never equivalent.
\end{defn}

\begin{defn}[The topology on $\bar{X}$]
\label{defn:BoundaryTopology}
Let $X$ be a proper $\delta$--hyperbolic space with basepoint $p \in X$.
Let $D > 2\delta$ be a fixed constant.
For each geodesic ray $c_0$ in $X$ with $c_0(0)=p$ we define $V(c_0,r,D)$ to be the set of equivalence classes of generalized rays that have a representative $c$
such that $c(0)=p$ and $d\bigl( c(r),c_0(r) \bigr) < D$.

According to \cite[Lem.~III.H.3.6]{BH99}, 
the sets $V(c_0,r,D)$ determine a topology as follows.
For each eventually constant generalized ray $c_0$ based at $p$, consider the family $\mathcal{V}_{c_0}$ of metric balls about $c_0(\infty)$.
For each geodesic ray $c_0$ based at $p$ and each $D>2\delta$, consider the family $\mathcal{V}_{c_0}$ of sets $V(c_0,r,D)$ for $r<\infty$.
The systems of sets $\mathcal{V}_{c_0}$ determine a topology on $\bar{X}$ in which each $\mathcal{V}_{c_0}$ is a neighborhood base at $c_0(\infty)$.
We caution the reader that the basic neighborhoods $V(c_0,r,D)$ are not necessarily open in this topology.

The topology on $\bar{X}$ does not depend on the chosen basepoint or the choice of constant $D>2\delta$.
We refer the reader to \cite[\S III.H.3]{BH99} for more details and proofs of the above assertions.
\end{defn}

Each basic neighborhood $V(c_0,r,D)$ in $\bar{X}$ is path connected. Indeed, choose a pair of generalized rays $c,c'$ based at $p$ that intersect the ball of radius $D$ about $c_0(r)$.  Then $c(\infty)$ can be joined to $c'(\infty)$ by a piecewise geodesic path
\[
   c\bigl( [r,\infty] \bigr) \cup
   \bigl[ c(r),c_0(r) \bigr] \cup 
   \bigl[ c_0(r), c'(r) \bigr] \cup
   c'\bigl( [r,\infty] \bigr)
\]
that lies inside $V(c_0,r,D)$.

\begin{defn}[Radial paths]
\label{defn:RadialPaths}
For each neighborhood $V(c_0,r,D)$ the paths described above will be referred to as \emph{radial paths}.  We have seen that any two points of $V(c_0,r,D)$ are joined by a radial path.
\end{defn}

\begin{defn}[Extended Gromov product]
\label{def:ExtendedGromov}
Let $X$ be a proper $\delta$--hyperbolic space with basepoint $p\in X$.  The Gromov product extends to $\bar{X}$ by defining
\[
   (x | y)_p = \inf \liminf_{i,j\to\infty} (x_i | y_j)_p,
\]
where the infimum is taken over all sequences $(x_i)$ and $(y_j)$ converging to $x$ and $y$ respectively.
For any sequences $x_i \to x$ and $y_i \to y$, we have
\[
   ( x | y )_p \le \liminf (x_i | y_i)_p \le \limsup (x_i | y_i)_p \le (x | y)_p + 2 \delta.
\]
See \cite[Lem.~2.2.2]{BuyaloSchroeder07}.
\end{defn}

\begin{lem}
\label{lem:TripodIdeal}
Let $X$ be a proper $\delta$--hyperbolic space with a basepoint $p$. Choose rays $c$ and $c'$ based at $p$.
For any $s \le \bigl(c(\infty) \big| c'(\infty) \bigr)_p$ we have $d \bigl( c(s),c'(s) \bigr) \le \delta$.
\end{lem}

\begin{proof}
For any $\epsilon>0$, there exists $T$ so that $s \le \bigl( c(T) \big| c'(T) \bigr)_p + \epsilon$.
By the definition of hyperbolicity, setting $s_0 = \bigl( c(T) \big| c'(T) \bigr)_p$ we have
\[
   d \bigl( c(s_0),c'(s_0) \bigr) \le \delta.
\]
There are two cases:
If $s\le s_0$ then $d \bigl( c(s),c'(s) \bigr) \le \delta$ by hyperbolicity.
Otherwise, $s \in [s_0, s_0 + \epsilon)$, so
\[
   d \bigl( c(s),c'(s) \bigr) \le \delta + 2\epsilon.
\]
Taking $\epsilon \to 0$, we get
$d \bigl( c(s),c'(s) \bigr) \le \delta$.
\end{proof}

\begin{rem}[Slim ideal triangles]
\label{rem:SlimIdeal}
A triangle is \emph{ideal} if one or more sides is either a ray or a bi-infinite line. If $X$ is $\delta$--hyperbolic, then each side of an ideal triangle lies in the $\delta_\infty$--neighborhood of the union of the other two sides, for some constant $\delta_\infty$ depending only on $\delta$.  See \cite[Prop.~2.2.2]{CDP90}.
\end{rem}

\begin{defn}[Visual metric]
\label{def:VisualMetric}
The boundary of a proper $\delta$--hyperbolic space $X$ admits a metric $d_\infty$ with the following property: there exists a parameter $a>1$ and constants $k_1,k_2>0$ such that
\[
   k_1 \, a^{-(\zeta | \xi)_p}
   \le d_\infty (\zeta,\xi) \le
   k_2 \, a^{-(\zeta|\xi)_p}
\]
holds for all $\xi,\zeta \in \boundary X$. (See \cite[Thm.~2.2.7]{BuyaloSchroeder07}.)
Such a metric $d_\infty$ is called a \emph{visual metric with parameter $a$} on $\boundary X$. 
\end{defn}

In the following proposition we relate the distance between certain points on a pair of rays on a given sphere with the distance between their endpoints in the visual metric.   This proposition will be used in the proof of Theorem~\ref{thm:LinearlyConnected}.

\begin{prop}
\label{prop:Uniformity}
Let $X$ be a $\delta$--hyperbolic space with basepoint $p$.
Let $d_\infty$ be a visual metric with parameter $a$ on $\boundary X$, and let $D$ be any fixed constant.
Then there exists a constant $k_3 = k_3(D)$ such that the following holds.
Choose any pair of geodesic rays $c,c'$ based at $p$.
For any $r>0$,
\[
   \text{if} \quad
   d \bigl( c(r),c'(r) \bigr) < D \quad
   \text{then} \quad
   d_\infty \bigl( c(\infty),c'(\infty) \bigr) < k_3 \, a^{-r}.
\]
\end{prop}

\begin{proof}
Given a pair of rays $c,c'$ based at $p$ and a constant $r>0$,
let $\xi=c(\infty)$ and $\xi'=c'(\infty)$.
The definition of the Gromov product implies that
\[
   \bigl( c(u) \big| c(v) \bigr)_p = \min \{u,v\}
\]
for any $u,v \in [0,\infty)$.

We will show that
\[
   \bigl( c(r) \big| c'(r) \bigr)_p > r - D/2.
\]
Consider the quantities:
\[
   a = \bigl( p \big| c'(r) \bigr)_{c(r)}
   \qquad\quad
   b = \bigl( p \big| c(r) \bigr)_{c'(r)}
   \qquad\quad
   s = \bigl( c(r) \big| c'(r) \bigr)_p
\]
Definition~\ref{def:DeltaHyperbolic} shows that
$a+s = r = b+s$, which implies $a=b$. Therefore,
$2a = a+b = d\bigl( c(r),c'(r) \bigr) <D$.
We conclude that
\[
   s = r-a > r - D/2.
\]
For any $\epsilon >0$ there exists $T>r$ such that
\begin{align*}
   ( \xi | \xi' )_p
   &\ge \limsup_{t\to\infty} \bigl( c(t) \big| c'(t) \bigr)_p -2 \delta
   \ge \bigl( c(T) \big| c'(T) \bigr)_p - 2\delta - \epsilon \\
   &\ge \min \bigl\{ \bigl(c(T) \big| c(r) \bigr)_p,
   \bigl(c(r) | c'(r) \bigr)_p, \bigl(c'(r) \big| c'(T) \bigr)_p \bigr\} -4\delta-\epsilon \\
   &\ge r - (D/2) -4\delta - \epsilon,
\end{align*}
where the third inequality follows by applying Proposition~\ref{prop:GromovInequality} twice.
Taking $\epsilon \to 0$, we see that
\[ 
( \xi | \xi')_p \ge r-(D/2)-4\delta.
\]
Therefore, we have
\[
   d_\infty ( \xi,\xi')
   \le k_2\, a^{-( \xi | \xi')_p}
   \le k_2\, a^{(D/2)+4\delta} a^{-r}.  \qedhere
\]
\end{proof}

\section{Separating horoballs and global cut points}
\label{sec:CutPoints}

This section introduces the property $(\ddag_M)$ and gives an elaboration of the proof of \cite[Prop.~3.3]{BestvinaMess91}.  In order to create paths in the boundary of a one-ended hyperbolic group $G$, Bestvina--Mess introduce $(\ddag_M)$, which may or may not hold in the Cayley graph of $G$. This property allows one to locally push paths outward away from a basepoint.  It turns out that the obstruction to $(\ddag_M)$ is the existence of a separating horoball in the Cayley graph.  Developing this idea, Bestvina--Mess show that if $(\ddag_M)$ fails to hold for any $M>0$, then the boundary must contain a cut point.
Recall that a \emph{cut point} of a connected topological space $Z$ is a point $p\in Z$ such that $Z - \{p\}$ is disconnected.
Subsequently, it was shown that the boundary of a one-ended hyperbolic group can never contain a cut point, which implies that $(\ddag_M)$ holds for every $M>0$.

\begin{defn}[Horoballs]
\label{def:Horoballs}
Let $X$ be a proper $\delta$--hyperbolic space, and let $\delta_\infty$ be the constant given by Remark~\ref{rem:SlimIdeal}.
A function $h \colon X \to \R$ is a \emph{horofunction} about $\eta\in \boundary X$ if the following condition holds:
There exists a constant $K<\infty$ such that if $a,b \in X$ and $d\bigl( b, [a,\eta) \bigr) \le \delta_\infty$ for some geodesic ray $[a,\eta)$, then $\bigabs{h(a)-h(b) - d(a,b)} \le K$.
A \emph{horoball} centered at $\eta\in \boundary X$ is a subset $B\subset X$ of the form $h^{-1}(-\infty,0]$ for some horofunction $h$ about $\eta$.
Each horoball has a unique center, which is the unique accumulation point of the horoball in $\boundary X$ (see \cite[\S 5]{BowditchRelHyp}).
\end{defn}

Gromov has considered a different, universal notion of horofunction defined on any proper metric space $X$ via a natural embedding of $X$ into the space of continuous real-valued functions on $X$ (see \cite{Gromov81}).  In the setting of $\delta$--hyperbolic spaces, the ``coarse'' definition above, due to Bowditch, is more general than Gromov's definition but is well-suited for coarse geometric arguments.
The following example illustrates how to use Gromov's construction to produce horoballs in the sense of Definition~\ref{def:Horoballs}.

\begin{exmp}
\label{exmp:LimitOfBalls}
A closed subset $A$ of a metric space $X$ is the \emph{Kuratowski--Painlev\'{e} limit} of a sequence of subsets $A_n$, denoted $A=\lim_{n\to\infty} A_n$, if
\[
   A = 
   \bigset{a\in X}{\limsup_{n\to\infty} d(a,A_n) = 0}
   =
   \bigset{a\in X}{\liminf_{n\to\infty} d(a,A_n) = 0}
   .
\]
Every sequence of subsets of a separable metric space has a subsequence that is convergent in the above sense by \cite[\S 29.VIII]{Kuratowski_VolI}.

Let $X$ be a proper $\delta$--hyperbolic space with basepoint $y$, and let $D\in \R$ be a constant.
By \cite[\S 2]{KarlssonMetzNoskov_06}, every sequence $(x_n)$ converging to a point $\eta\in \boundary X$ has a subsequence such that the balls $B(x_{n_i}, d(y,x_{n_i}) + D)$ converge to a horoball centered at $\eta$.
\end{exmp}

\begin{defn}
\label{def:AlmostExtendable}
Let $X$ be a geodesic space.
We say that \emph{geodesics are almost extendable} in $X$ if there is a constant~$C$ such that for all points $p$ and~$q$ in~$X$ there is a geodesic ray emanating from~$p$ that passes within a distance~$C$
of~$q$.
We typically say that $X$ is \emph{almost extendable} for short.
The constant $C$ is an \emph{almost extendability constant} of $X$.
In the literature, almost extendable spaces are sometimes called \emph{visual} spaces (see, for instance, \cite{BonkSchramm00}).
\end{defn}

Mihalik observed that infinite word hyperbolic groups are almost extendable (see \cite[\S 3]{BestvinaMess91}).
We first review the simple proof of this fact.

\begin{prop}
\label{prop:AlmostExtendable}
Let $X$ be any proper $\delta$--hyperbolic space with a cocompact group action by a group $G$.
If $\boundary X$ has at least two points, then $X$ is almost extendable.
\end{prop}

\begin{proof}
The space $X$ contains a bi-infinite geodesic line. Choose a compact set $K$ whose $G$--translates cover $X$, and let $D = \diam(K)$.
Choose points $p,q \in X$.
By cocompactness, there exists a geodesic line $c$ intersecting the closed ball of radius $D$ centered at $q$.
Consider the ideal geodesic triangle $\Delta$ having $p$ as a vertex and $c$ as a bi-infinite side.
By Remark~\ref{rem:SlimIdeal}, the triangle $\Delta$ is $\delta_\infty$--slim.
Note that $q$ is at most a distance $D$ from some point $c(t)$ on the line $c$, and $c(t)$ is within a distance $\delta_\infty$ of one of the rays $\bigl[ p, c(\infty)\bigr)$ or $\bigl[ p,c(-\infty) \bigr)$.
Therefore the ray in question comes within a distance $D+\delta_\infty$ of the point $q$.
\end{proof}

Throughout this section, $X$ is a one-ended $\delta$--hyperbolic space that admits a geometric group action and $C$ is its almost extendability constant.

In \cite{BestvinaMess91}, Bestvina--Mess introduce the following property ($\ddag_M$), which may or may not be satisfied by $X$ for each constant $M>0$.  The ($\ddag_M$) property is key for establishing the connection between local connectivity and separating horoballs.
\begin{description}
\item[($\ddag_M$)]
There is a number $L > 0$ such that for every $R>0$ and any three
points $x$,~$y$, $z \in X$ with $d(x,y) = d(x,z) = R$ and
$d(y,z) \le M$, there exists a path of length at most~$L$
connecting $y$ to~$z$ in the complement of the ball of radius
$R-C$ centered at~$x$.
\end{description}

The main goal of this section is to prove Corollary~\ref{cor:ddagHolds}, which states that $(\ddag_M)$ holds for all $M>0$.

A horoball $B$ in $X$ is \emph{essentially separating} if $B$ is closed in $X$ and the complement of $B$ has at least two components that each contain a geodesic ray not asymptotic to $B$.

\begin{prop}
\label{prop:SeparatingHoroball}
Let $X$ be a proper $\delta$--hyperbolic space with a cocompact group action by $G$.
Suppose there exists a constant $M>0$ such that $(\ddag_M)$ fails in $X$.  If $X$ is one-ended then it contains an essentially separating horoball.
\end{prop}

The proof of this result depends on the almost extendable property of $X$, but hyperbolicity plays no essential role in the proof. We note that any proper $\CAT(0)$ space with a cocompact group of isometries is almost extendable by \cite{GeogheganOntaneda07}.

\begin{proof}
If $(\ddag_M)$ fails, then for each $n$ there is a positive number $R_n$ and a triple of points
$x_n$, $y_n$, and $z_n$ with $d(x_n,y_n) = d(x_n,z_n) = R_n$ and $d(y_n,z_n) \le M$ such that any path from $y_n$ to $z_n$ in the complement of the closed ball $B_n = \bar{B}(x_n,R_n-C)$ has length at least~$n$.  Since $X$ is almost extendable, we can choose rays $c_n$ and $c'_n$ based at $x_n$ passing within a distance $C$ of $y_n$ and $z_n$ respectively.

After passing to a subsequence, we can assume that either $R_n \to \infty$ or $R_n$ remains bounded as $n \to \infty$.  Suppose first that $R_n$ remains bounded.  After passing to a further subsequence and applying the cocompact isometry group, we can assume that $x_n, y_n$, and $z_n$ converge to points $x,y,z \in X$ respectively.  We can also assume that $c_n$ converges to a ray from $x$ to some point $\xi \in \boundary X$ that passes within $C$ of $y$.  Similarly, $c'_n$ converges to a ray from $x$ to some $\xi'\in \boundary X$ that passes within $C$ of $z$.  In this case, the closed balls $B_n$ converge, in the sense of Kuratowski--Painlev\'{e}, to a compact set $B$.  
The existence of a path $\alpha$ of finite length $< L$ from $y$ to $z$ in the complement of $B$ would imply the existence of a path $\alpha'$ from $y_n$ to $z_n$ in the complement of $B_n$ of length $<L$ for sufficiently large $n$.  However, there is no such path when $n>L$, so $B$ must separate the ends of $X$ containing $\xi$ and $\xi'$.  In particular these ends are distinct, which is impossible since $X$ is one-ended.

Therefore the sequence $R_n$ must limit to infinity.  As before we may pass to a subsequence so that $y_n,z_n$ converge to $y,z \in X$ and $x_n \to \eta \in \boundary X$ and the balls $B_n$ converge to a closed horoball $B$ centered at $\eta$ (see Example~\ref{exmp:LimitOfBalls}). In this case $y$ and $z$ must lie in the complement of $B$.  We may also assume that the rays $c_n$ converge to a geodesic line from $\eta$ to $\xi$ that passes within a distance $C$ of $y$. Similarly the rays $c'_n$ converge to a line from $\eta$ to some $\xi'$ that passes within $C$ of $z$.  By the same reasoning as above, there is no path of finite length from $y$ to $z$ in the complement of $B$. As in the previous case, $B$ is an essentially separating horoball.
\end{proof}

The hyperbolicity of the space $X$ plays a key role in the next result, which does not require a cocompact group action.

\begin{prop}\label{prop:CutSet}
Suppose $X$ is a one-ended proper $\delta$--hyperbolic space containing an essentially separating horoball $B$ centered at $\eta \in \boundary X$.
Then $\eta$ is a cut point of $\boundary X$.
\end{prop}


\begin{proof}
Fix a basepoint $p \in X$.
Let $\eta\in \boundary X$ be the center of the horoball $B$. Choose rays $c$ and $c'$ based at $p$, such that they have tails in distinct components of $X \setminus B$.  Let $\xi = c(\infty)$ and $\xi'=c'(\infty)$.
We will show that $\xi$ and $\xi'$ are in distinct components of $\boundary X \setminus \{\eta\}$.

Fix a constant $D >2\delta$, and consider the corresponding neighborhood base for points of $\boundary X$ described in Definition~\ref{defn:BoundaryTopology}.
Since $\bar{B} = B \cup \{\eta\}$ is closed in $\bar{X}$, each point $\zeta \in \boundary X \setminus \{\eta\}$ has a basic neighborhood $W$ in $\bar{X}\setminus \bar{B}$.
Recall that basic neighborhoods are not necessarily open, but their points may be connected by radial paths, as described in Definition~\ref{defn:RadialPaths}.
Let $\Set{U}$ be the set of all interiors of such neighborhoods $W$.  Then $\boundary{X} \setminus \{\eta\}$ is contained in the union of all members of $\Set{U}$.

Suppose by way of contradiction that $\xi$ and $\xi'$ are not separated by $\eta$ in $\boundary X$.
Then we may choose a chain $U_1, \dots, U_k$ in~$\Set{U}$ such that
$\xi \in U_1$, $\xi' \in U_k$, and $U_i \cap U_{i+1} \cap \bigl(\boundary X \setminus \{\eta\}\bigr)$ is nonempty.
Choose $a_0$ to be a point on the ray $c$ that lies inside $U_1$. Similarly choose $a_k$ on $c'$ lying inside $U_k$.
For $0<i<k$, choose $a_i \in U_i \cap U_{i+1}$.
For each $i$, choose a basic neighborhood $W_i \subseteq \bar{X} \setminus \bar{B}$ such that $U_i \subseteq W_i$.
Choose a radial path $\gamma_i$ from $a_{i-1}$ to $a_{i}$ lying inside $W_i$.
Since each $a_i$ lies in the interior of a basic neighborhood, we may assume that the paths $\gamma_i$ and $\gamma_{i+1}$ use the same generalized ray from $p$ to $a_i$.
The union $\gamma_1 \cup \cdots \cup \gamma_k$ is a connected set $K$ in $\bar{X}\setminus \bar{B}$ containing $a_0$ and $a_k$.
Inside $K$ is a path of $X \setminus B$ joining $a_0$ to $a_k$, contradicting the choice of $c$ and $c'$.
\end{proof}

We now use the Cut Point Theorem, which asserts the boundary of a one-ended hyperbolic group does not contain a cut point.  The Cut Point Theorem combines major contributions of Bowditch, Levitt, and Swarup \cite{Bowditch99Treelike,Levitt98,Swarup96}.

\begin{cor}
\label{cor:ddagHolds}
Suppose $X$ is a one-ended $\delta$--hyperbolic space and $G$ acts geometrically on $X$. Then $(\ddag_M)$ holds for all $M>0$.
\end{cor}

\begin{proof}
By the Cut Point Theorem, the boundary of $X$ does not contain a cut point.  Suppose by way of contradiction that $(\ddag_M)$ fails for some $M>0$.  
Then by Proposition~\ref{prop:SeparatingHoroball}, the space $X$ contains an essentially separating horoball $B$.
Finally, by Proposition~\ref{prop:CutSet} shows that any such horoball gives rise to a cut point in $\boundary X$, which contradicts the cut point theorem.
\end{proof}

\section{Semistability at infinity}
\label{sec:Semistable}

The goal of this section is to give a simple, direct proof of semistability for one-ended hyperbolic groups using the condition $(\ddag_M)$. Indeed, $(\ddag_M)$ is very close in spirit to a characterization of semistability studied by Mihalik.
We also explain how to derive semistability for arbitrary hyperbolic groups from the one-ended case.

Let $X$ be a simply connected space.
A \emph{ray} in $X$ is a proper map $c\colon [0,\infty) \to X$.
Two rays $c$ and $c'$ in $X$ \emph{converge to the same end} of $X$ if for any compact set $D \subset X$ there exists $N\ge 0$ such that $c\bigl( [N,\infty) ] \bigr)$ and $c\bigl( [N,\infty) \bigr)$ lie in the same component of $X\setminus D$.  Convergence to the same end is an equivalence relation whose equivalence classes are called the \emph{ends} of $X$.

A simply connected space $X$ is \emph{semistable at infinity} if any two rays in $X$ converging to the same end are joined by a proper homotopy.

A theorem of Mihalik (\cite[Thm.~2.1]{Mihalik83}) states that semistability is equivalent to the following condition:
for some (equivalently any) ray $r$ and each compact set $D \subseteq X$, there is a larger compact set $E$ so that, for every still larger compact set $F$, each pointed loop $\alpha$ in  $X \setminus E$ based on the ray $r$ can be homotoped into $X\setminus F$ via a homotopy in $X-D$ that slides the base point along $r$.

The following lemma makes precise the idea that $(\ddag_M)$ allows one to push paths outward in the Cayley graph in a controlled manner.

\begin{lem}
\label{lem:PushPaths}
Let $G$ be a one-ended word hyperbolic group acting geometrically on a $\delta$--hyperbolic space $X$.
There is a constant $\lambda$ such that the following holds.
Choose any basepoint $p\in X$.
Let $c$ and $c'$ be geodesic rays based at $p$.

If there exists a path $\alpha$ of length at most $\ell$ from $c$ to $c'$ in the complement of the ball $B(p,r)$ then there exists a path $\beta$ of length at most $\ell\lambda$ from $c$ to $c'$ in the complement of $B(p,r+1)$.

If $X$ is simply connected, there is a further constant $\kappa=\kappa(X)$ such that the following holds.
If $r>r_0+\kappa$ for some positive constant $r_0$, then for any path $\alpha$ as above, the path $\beta$ may be chosen homotopic to $\alpha$ by a homotopy consisting of paths from $c$ to $c'$ that lie in the complement of $B(p,r_0)$.
\end{lem}

\begin{proof}
Let $C$ be the almost extendability constant of $X$, and let $M=6C+3$.  Consider the constant $L=L(M)$ given by $(\ddag_M)$.
Choose any radius $r\ge 0$.
Consider a path $\alpha\colon [0,\ell] \to X$ in the complement of the ball $B(p,r)$ parameterized so that $d \bigl( \alpha(j-1),\alpha(j) \bigr) \le 1$ for each $j = 1,\dots,\ell$.
Since $X$ is almost extendable, for each $j=0,\dots,\ell$ we may choose 
a geodesic ray $c_j$ based at $p$ containing a point $y_j$ such that $d\bigl(y_j,\alpha(j)\bigr) < C$.
We may assume that $c_0=c$ and $c_\ell=c'$.
Then for each $j = 0,\dots,\ell$ the point $y_j$ lies outside the ball $B(y,r-C)$.
Choose points $z_0, \dots, z_\ell$ on these rays so that
$d(y_j,z_j)$ is at most $2C+1$ and $d(p,z_j)$ is at least $r+C+1$.
Then $d(z_{j-1},z_{j})$ is at most $6C + 3=M$.  By $(\ddag_M)$, there exists a path $\beta_j$ of length at most $L$ joining $z_{j-1}$ to $z_{j}$ in the complement of the ball $B(p,r+1)$.
Since the path $\beta$ formed by concatenating the paths $\beta_j$ has length at most $\ell L$, the first conclusion of the lemma follows if we set $\lambda=L$.

Now suppose $X$ is simply connected.  By cocompactness, there is a constant $H$ so that any loop $\gamma$ in $X$ of length at most $M+L$ bounds a disc in the $H$--neighborhood of $\gamma$.
Let $\kappa = C+H$.
For any fixed radius $r_0$, suppose now that $r>r_0+\kappa$.
Let $\eta_j$ be a geodesic from $\alpha(j)$ to $y_j$. Consider a loop $\gamma_j$ formed by following $\alpha$ from $\alpha(j-1)$ to $\alpha(j)$, then following $\eta_j$ to $y_j$, following $c_j$ from $y_j$ to $z_j$, following $\beta_j$ to $z_{j-1}$, following $c_{j-1}$ to $y_{j-1}$, and then following $\eta_{j-1}$ back to $\alpha(j-1)$.
Since $\gamma_j$ is a loop of length at most $M+L$ in the complement of $B(p,r-C)$, it bounds a disc in the complement of the ball $B(p,r-\kappa)$.  In particular $\gamma_j$ bounds a disc disjoint from $B(p,r_0)$.
\end{proof}

\begin{prop}
\label{prop:OneEndedSemistable}
Every one-ended word hyperbolic group is semistable at infinity.
\end{prop}

\begin{proof}
Assume $G$ acts geometrically on a simply connected $\delta$--hyperbolic space $X$.
Choose a basepoint $p$ and a geodesic ray $c$ based at $p$.
We apply Lemma~\ref{lem:PushPaths} with $c=c'$.
Using the notation of Lemma~\ref{lem:PushPaths}, a loop $\alpha$ outside the ball $B(p,r)$ can be homotoped to a path in the complement of $B(p,r+1)$ by a homotopy in the complement of $B(p,r_0)$.
The homotopy involves paths with each of their endpoints on $c$.  At each stage of the homotopy, we insert a subpath of $c$ joining these endpoints so that we now have a homotopy of loops.
Iterating this conclusion, we may homotope loops into the complement of an arbitrarily large ball to see that $X$ is semistable at infinity.
\end{proof}

We now discuss how to pass from the one-ended case to the general case to conclude the following result, first observed by Geoghegan using a different argument (see \cite{BestvinaMess91}).

\begin{thm}[Geoghegan]
\label{thm:Semistable}
Every word hyperbolic group $G$ is semistable at infinity.
\end{thm}

This theorem follows from Proposition~\ref{prop:OneEndedSemistable} using a result of Mihalik that establishes semistability for graphs of groups over finite edge groups if the vertex groups are known to be semistable \cite{Mihalik_OneEnded}.
In order to keep our explanation relatively self-contained, we sketch a proof of Theorem~\ref{thm:Semistable} below.  In particular, we sketch Mihalik's reduction from the general case to the one-ended case in the setting of hyperbolic groups.

The first ingredient in the proof of this theorem provides a useful graph of groups decomposition by combining several well-known results.

\begin{prop}
\label{prop:Combination}
Let $G$ be a word hyperbolic group. Then $G$ is the fundamental group of a graph of groups whose edge groups are finite and whose vertex groups are either finite or one ended.
Furthermore, each vertex group $G_v$ is itself a hyperbolic group, and the inclusion $G_v \hookrightarrow G$ is a quasi-isometric embedding.
\end{prop}

\begin{proof}
Every hyperbolic group $G$ is finitely presented, so Dunwoody's accessibility theorem implies that $G$ splits over finite groups as described \cite{Dunwoody85}.
It follows that each vertex group $G_v$ is quasiconvex, in the sense that, any geodesic in the Cayley graph of $G$ joining two points of $G_v$ must lie uniformly close to $G_v$ itself.
Any quasiconvex subgroup of a hyperbolic group is hyperbolic and quasi-isometrically embedded.  See, for instance, Lemma III.$\Gamma$.3.5 and Proposition III.$\Gamma$.3.7 of \cite{BH99}.
\end{proof}

To build a cleaner geometric model space for a multi-ended hyperbolic group, we use the following construction.

\begin{prop}
\label{prop:FixedPoint}
Let $G$ be a word hyperbolic group.  There exists a simply connected $\delta$--hyperbolic simplicial complex $X$ on which $G$ acts geometrically such that each finite subgroup of $G$ has a fixed point in $X$.
\end{prop}

\begin{proof}
Let $d_S$ be the word metric for a finite generating set $S$.
Since a hyperbolic group has finitely many conjugacy classes of finite subgroups \cite[III.$\Gamma$.3.2]{BH99}, there exists a constant $R$ so that each finite subgroup has a conjugate of diameter less than $R$.  For sufficiently large $R$, each finite subgroup stabilizes a simplex in the Rips complex $X=P_R(G,d_S)$ and, thus, the subgroup has a fixed point in $X$.
\end{proof}

An entirely different proof of Proposition~\ref{prop:FixedPoint} is given by Lang, who shows that the injective hull of $(G,d_S)$ is a contractible complex satisfying the conclusion of the preceding result in \cite{Lang_InjectiveHull}.

We are now ready to sketch the proof of the theorem.

\begin{proof}[Proof of Theorem~\ref{thm:Semistable}]
Let the infinite hyperbolic group $G$ be the fundamental group of a graph of groups $\mathcal{G}$ with all edge groups finite and all vertex groups either finite or one-ended.
Let $T$ be the Bass--Serre tree for the given splitting.
For each one-ended vertex group $G_v$, let $X_v$ be a simply connected $\delta$--hyperbolic simplicial complex on which $G_v$ acts geometrically such that each finite subgroup of $G_v$ has a fixed point, as in Proposition~\ref{prop:FixedPoint}. We assume that these spaces and fixed points are chosen equivariantly.
For each edge $e$ of $T$ joining vertices $v$ and $w$, let $p_e \in X_v$ and $q_e \in X_w$ be the chosen points fixed by the edge group $G_e$.
The group $G$ acts geometrically on a $\delta$--hyperbolic space $X$ formed by gluing the spaces $\{X_v\}$ in the pattern of the tree $T$ as follows:  Start with the disjoint union of the spaces $X_v$ for $v \in T$, and for each edge of $T$, attach an edge with endpoints $p_e$ and $q_e$. We will call these new edges the $T$--edges of $X$.

We will see that any pair of rays in $X$ converging to the same end are properly homotopic.
Let $\pi\colon X \to T$ be the quotient map that collapses each space $X_v$ to a point, leaving only the $T$--edges.
Recall that a map is \emph{monotone} if the preimage of each point is a connected set.
We claim that each proper ray in $X$ is properly homotopic to a ray that projects to a monotone path in $T$.  Indeed, let $r$ be any ray of $X$ that is also an edge path, \emph{i.e.}, traversing each edge linearly. For each point $p$ in a closed $T$--edge, the complement $X\setminus\{p\}$ has exactly two components that we call \emph{halfspaces}, each of which is simply connected.
If $H$ is a halfspace not containing the basepoint $r(0)$, then any subpath of $r$ forming a loop in $\bar{H}$ based at $p$ can be contracted to a point within $\bar{H}$.
Thus we may assume that $r$ contains no such loop.  But a ray containing no such loop projects to a monotone path $\bar{r} = \pi\of r$ in $T$.

Either $\bar{r}$ is eventually constant (so that $r$ has a tail lying inside a single vertex space $X_v$) or $\bar{r}$ converges to an end of $T$ (so that $r$ travels through an infinite number of vertex spaces, but never visits any of them more than once).

If two rays with monotone projection converge to the same end of $X$, they must have the same type in $T$---either their projections to $T$ have tails that lie in the same vertex space $X_v$ or their projections converge to the same end of $T$.  In the first case, they are properly homotopic since, by Proposition~\ref{prop:OneEndedSemistable}, the space $X_v$ is semistable.
In the second case, they are properly homotopic, since each vertex space is simply connected.  
\end{proof}

\section{Linear connectivity}
\label{sec:LinearCon}

In this section, we prove two results due to Bonk--Kleiner on linear connectivity of boundaries and the existence of quasi-isometrically embedded hyperbolic planes in hyperbolic groups.  The proof of linear connectivity is a slight variation of the elementary proof of \cite[Prop.~3.2]{BestvinaMess91}.

\begin{defn}
A complete metric space $M$ is \emph{linearly connected} if there exists a constant $\nu<\infty$ such that each pair of points $x,x'\in M$ is contained in a compact, connected set of diameter less than $\nu\,d(x,x')$.
\end{defn}

Given two points $x,x'\in M$, a \emph{$\rho$--chain of length $N$ from $x$ to $x'$} is a sequence of points $x=b_0,b_1,\dots,b_N=x'$ such that $d(b_{j-1},b_j) \le \rho \, d(x,x')$ for each $j$.

\begin{prop}
\label{prop:GeometricSeries}
Let $M$ be a compact metric space. Suppose there are numbers $\rho<1$ and $N \in \N$ such that for each $x,x'\in M$ there is a $\rho$--chain of length $N$ from $x$ to $x'$.
Then $M$ is linearly connected.
\end{prop}

\begin{proof}
Given $x,x' \in M$, let $S_1$ be a $\rho$--chain of length $N$ from $x$ to $x'$.
Define $S_{k}$ recursively by inserting a $\rho$--chain of length $N$ between each consecutive pair of points of $S_{k-1}$.
Then
\[
   \diam(S_k) \le 2N(\rho + \rho^2 + \cdots + \rho^k) \, d(x,x')
\]
holds by induction, since
\[
   \diam(S_1) \le N\rho\,d(x,x')
   \quad \text{and} \quad
   \diam(S_k) \le \diam(S_{k-1}) + 2N \rho^k\,d(x,x').
\]
Therefore, $S = \overline{\bigcup S_k}$ is a compact connected set of diameter at most $\nu \, d(x,x')$ for $\nu=2N\rho /(1-\rho)$ because $\rho < 1$.
\end{proof}

\begin{rem}
The proposition above also holds under the weaker hypothesis that $M$ is a complete metric space.
Indeed, the inclusion $\bigcup S_k \to M$ determines a map $[0,1] \cap \Z[1/N] \to M$, which is uniformly continuous by a slight variation of the above argument. By completeness, this map extends to a continuous path $I \to M$ from $x$ to $x'$ with image $S$.
Curiously, Hilbert's well-known construction of a space-filling curve may be recovered as a direct consequence: if chains $S_k$ are selected in the square $I^2$ according to Hilbert's recursive method, the result is a continuous surjection $I \to S=I^2$.
\end{rem}

\begin{thm}[Bonk--Kleiner]
\label{thm:LinearlyConnected}
If $G$ is any one-ended hyperbolic group, then $\boundary G$ is linearly connected with respect to any visual metric.
\end{thm}

\begin{cor}[\cite{BestvinaMess91,Bowditch99Treelike,Levitt98,Swarup96}]
The boundary of a one-ended hyperbolic group is locally connected. \qed
\end{cor}

\begin{proof}[Proof of Theorem~\ref{thm:LinearlyConnected}]
Suppose $G$ acts geometrically on a $\delta$--hyperbolic space $X$.
Fix a basepoint $p \in X$, and choose a visual metric $d_\infty$ on $\boundary X$ with parameter $a$.
Choose distinct points $x,x' \in \boundary X$, and let $s = (x | x')_p$. If $c$ and $c'$ are any rays based at $p$ such that $c(\infty)=x$ and $c'(\infty)=x'$, then $d \bigl( c(s),c'(s) \bigr) \le \delta$ by Lemma~\ref{lem:TripodIdeal}.
Let $C$ be the almost extendability constant of $X$, and let $L=L(\delta)$ be the constant given by $(\ddag_M)$ using $M=\delta$.
Then there exists a path $\alpha$ of length at most $L$ from $c(s)$ to $c'(s)$ in the complement of the ball $B(p,s-C)$.

Let $m>0$ be a positive integer to be determined later, and let $\lambda$ be the constant from Lemma~\ref{lem:PushPaths}.
If we iterate the lemma $m$ times, we see that there exists a path $\beta$ of length at most $N=\lambda^m L$ from $c$ to $c'$ in the complement of the ball $B(p,s - C + m)$.
Fix a parametrization $\beta\colon [0,N] \to X$ so that $d\bigl( \beta(j-1),\beta(j) \bigr) \le 1$ for each $j=0,\cdots,N$.
Since $X$ is almost extendable, for each $j$ we may choose a ray $c_j$ containing a point $y_j$ such that $d\bigl( y_j,\beta(j) \bigr) <C$.  We will assume that $c_0=c$ and $c_L = c'$.
Then for each $j$ we have
\[
   d(p,y_j) \ge s-2C+m
   \qquad \text{and} \qquad
   d(y_{j-1},y_{j}) \le 2C + 1.
\]
If we let $z_j= c_j(s-2C+m)$, then we have $d(z_{j-1},z_j) \le 2C +1 + 2\delta$. Let $k_1$ be given by Definition~\ref{def:VisualMetric}, and let $k_3 = k_3(2C+1+2\delta)$ be the constant from Proposition~\ref{prop:Uniformity}.
Setting $b_j = c_j(\infty) \in \boundary X$, we have
\[
   d_\infty (b_{j-1},b_j )
   \le k_3 a^{-(s-2C+m)} 
   = \frac{k_3}{a^{m-2C}} \, a^{-(x| x')_p }
   \le \frac{k_3}{k_1 a^{m-2C}}\, d_\infty (x,x').
\]
In particular, there exists a $\rho$--chain of length $N$ from $x$ to $x'$ for $\rho=k_3 / ( k_1 a^{m-2C} )$.
We now fix a value for the parameter $m$ large enough that $a^m > k_3 a^{2C}/k_1$.
Then $\rho<1$, and $N=\lambda^m L$ does not depend on the choice of $x$ and $x'$.
Thus $(\boundary X,d_\infty)$ is linearly connected by Proposition~\ref{prop:GeometricSeries}.
\end{proof}

\begin{thm}[Bonk--Kleiner]
\label{thm:QuasiHyperbolic}
Let $G$ be any hyperbolic group that is not virtually free. Then $G$ contains a quasi-isometrically embedded hyperbolic plane.
\end{thm}

The proof of Theorem~\ref{thm:QuasiHyperbolic} involves the following notions from metric analysis.  A metric space $M$ is \emph{doubling} if there is a constant $N$ such that each metric ball can be covered by at most $N$ balls of half the radius.
A nonconstant map $f\colon M \to M'$ between metric spaces is \emph{quasisymmetric} if there exists a homeomorphism $\eta\colon [0,\infty) \to [0,\infty)$ such that $d(x,y)\le t\,d(x,z)$ implies $d\bigl( f(x),f(y) \bigr) \le \eta(t) \, d\bigl( f(x),f(z) \bigr)$ for all $x,y,z \in M$ and all $t\ge 0$.
A \emph{quasi-arc} in $M$ is a quasisymmetric map $[0,1]\to M$.

Now that we have established the linear connectivity of the boundary, the proof of Theorem~\ref{thm:QuasiHyperbolic} is essentially the same as the proof in \cite{BonkKleiner05}.
It depends on several results not discussed in detail in this article.  The most difficult of these to establish is a delicate result on the existence of quasi-arcs, originally due to Tukia \cite{Tukia96}, for which a substantially simpler proof was later given by Mackay \cite{Mackay_Quasiarcs}.  The version proved by Mackay is more suited for direct use in the proof below.

\begin{proof}[Proof of Theorem~\ref{thm:QuasiHyperbolic}]
By Proposition~\ref{prop:Combination} the group $G$ splits as a graph of groups such that each vertex group is a quasi-isometrically embedded hyperbolic group with at most one end.
If all vertex groups were finite, then $G$ would be virtually free \cite[Thm.~7.3]{ScottWall79},
so at least one vertex group must be one ended.
Therefore, it suffices to prove the theorem when $G$ is one-ended.

The boundary of $G$ is linearly connected by Theorem~\ref{thm:LinearlyConnected}.  A theorem of Bonk--Schramm states that the boundary of any word hyperbolic group is doubling \cite[Thm.~9.2]{BonkSchramm00}.
In a doubling, linearly connected, compact metric space, any two points can be joined by a quasi-arc (\cite{Mackay_Quasiarcs}).
If $S$ is a sector of the hyperbolic plane, we thus have a quasisymmetric embedding of Gromov boundaries $\boundary S \to [0,1] \to \boundary G$.
Since $\boundary S$ is connected, the quasi\-symmetric embedding extends to a quasi-isometric embedding $S \to G$ as explained in \cite[Thms.\ 7.4 and~8.2]{BonkSchramm00}.
Since $S$ contains hyperbolic discs of arbitrarily large radius, a diagonal argument gives a quasi-isometric embedding of the entire hyperbolic plane into $G$ (see \cite[Lem.~II.9.34]{BH99}).
\end{proof}

\bibliographystyle{alpha}
\bibliography{hyperbolic.bib}

\end{document}